\documentclass[10pt, reqno]{amsart}
\usepackage{amssymb,amsmath,amsthm,color, cite}
\theoremstyle{plain}
\newtheorem{theorem}{Theorem}[section]
\newtheorem{proposition}[theorem]{Proposition}

\newtheorem{lemma}[theorem]{Lemma}

\theoremstyle{remark}
\newtheorem{remark}[theorem]{Remark}

\newcommand{\eps}{\varepsilon} 
\usepackage
[bookmarks=true,
   bookmarksnumbered=false,
   bookmarkstype=toc]
{hyperref}

\numberwithin{equation}{section}

\newcommand{\C}{\mathbb{C}}
\newcommand{\R}{\mathbb{R}}

\renewcommand{\Im}{\operatorname{Im}}
\renewcommand{\Re}{\operatorname{Re}}

\newcommand{\qtq}[1]{\quad\text{#1}\quad}

\begin{document}

\title[2d focusing NLS]{Scattering below the ground state for the 2$d$ radial nonlinear Schr\"odinger equation}

\author[A. Arora]{Anudeep Kumar Arora}
\address{Department of Mathematics \& Statistics, Florida International University, Miami, FL, USA}
\email{ana001@fiu.edu}

\author[B. Dodson]{Benjamin Dodson}
\address{Department of Mathematics, Johns Hopkins University, Baltimore, MD, USA}
\email{bdodson4@jhu.edu}

\author[J. Murphy]{Jason Murphy}
\address{Department of Mathematics \& Statistics, Missouri University of Science \& Technology, Rolla, MO, USA}
\email{jason.murphy@mst.edu}

\begin{abstract} We revisit the problem of scattering below the ground state threshold for the mass-supercritical focusing nonlinear Schr\"odinger equation in two space dimensions.  We present a simple new proof that treats the case of radial initial data.  The key ingredient is a localized virial/Morawetz estimate; the radial assumption aids in controlling the error terms resulting from the spatial localization. \end{abstract} 
\maketitle

\section{Introduction}

We consider the initial-value problem for the focusing nonlinear Schr\"odinger equation (NLS) in two space dimensions:
\begin{equation}\label{nls}
\begin{cases} (i\partial_t+\Delta) u = -|u|^p u, \\ u(0)=u_0\in H^1(\R^2),\end{cases}
\end{equation}
where $u:\R\times\R^2\to\C$ and $2<p<\infty$.  This equation admits a global nonscattering solution of the form $u(t)=e^{it}Q(x)$, where $Q$ is the \emph{ground state} solution to the elliptic equation
\[
-\Delta Q + Q - Q^{p+1} = 0. 
\]
In this note we will give a simple new proof of scattering for radial solutions to \eqref{nls} with initial data `below the ground state threshold'.  Before stating the result precisely, let us introduce a few basic notions.

First, we recall that solutions to \eqref{nls} conserve their \emph{mass} and \emph{energy}, defined by
\begin{align*}
M(u(t)) & = \int_{\R^2} |u(t,x)|^2\,dx, \\
E(u(t)) & = \int_{\R^2} \tfrac12|\nabla u(t,x)|^2 - \tfrac{1}{p+2}|u(t,x)|^{p+2}\,dx,
\end{align*}
respectively. 

Next, we observe that the class of solutions to \eqref{nls} is invariant under the scaling
\begin{equation}\label{scaling}
u(t,x)\mapsto \lambda^{\frac{2}{p}}u(\lambda^2 t,\lambda x). 
\end{equation}
This defines a notion of \emph{criticality} for \eqref{nls}.  Specifically, if we define 
\[
s_c=1-\tfrac{2}{p},
\]
then we find that the $\dot H^{s_c}$ norm of initial data is invariant under the rescaling \eqref{scaling}. In particular (since we are working in two space dimensions) we always have $s_c<1$, which means the equation is always \emph{energy-subcritical}.  The case $s_c=0$ (i.e. $p=2$) is called the \emph{mass-critical} equation, since in this case the mass of solutions is left invariant under \eqref{scaling}.  

In this paper, we consider the mass-supercritical range $2<p<\infty$. We will prove the following.

\begin{theorem}\label{T} Let $2<p<\infty$.  Suppose $u_0\in H^1(\R^2)$ is spherically-symmetric and satisfies
\begin{equation}\label{sub1}
M(u_0)^{2}E(u_0)^{p-2}<M(Q)^{2}E(Q)^{p-2}
\end{equation}
and
\begin{equation}\label{sub2}
\|u_0\|_{L^2}^{2}\|\nabla u_0\|_{L^2}^{p-2} < \|Q\|_{L^2}^{2}\|\nabla Q\|_{L^2}^{p-2}.
\end{equation}
Then \eqref{nls} admits a unique, global-in-time solution $u$ with $u|_{t=0}=u_0$.  Furthermore, the solution $u$ \emph{scatters}, that is, there exist $u_\pm\in H^1(\R^2)$ such that
\[
\lim_{t\to\pm\infty}\|u(t)-e^{it\Delta}u_\pm\|_{H^1}=0. 
\]
\end{theorem}

\begin{remark} The powers of $M(u_0)$ and $E(u_0)$ are chosen so that the product scales like the critical $\dot H^{s_c}$-norm.  Here $e^{it\Delta}$ denotes the Schr\"odinger group, so that $e^{it\Delta}u_{\pm}$ are solutions to the linear Schr\"odinger equation.  The proof of Theorem~\ref{T} will also show that the solution $u$ obeys global space-time bounds of the form
\[
\|u\|_{L_{t,x}^{2p}(\R\times\R^2)} \leq C(M(u_0),E(u_0)). 
\]
\end{remark}

Theorem~\ref{T} result was previously established in \cite{AkahoriNawa, CFX, Gue}, who extended the arguments of \cite{HR, DHR}.  In fact, in \cite{AkahoriNawa, CFX, Gue} the same result is proven without the restriction to radial initial data.  In these works the authors proceed via the concentration-compactness approach to induction on energy.  The purpose of this note is to demonstrate a short and simple argument that suffices to handle the radial case; in particular, it avoids concentration-compactness entirely.  This extends our previous works \cite{DM1, DM2} to the two-dimensional setting, which often presents new challenges due to issues with Morawetz estimates and weaker dispersive estimates.  It is an interesting problem to find a simplified argument to handle the general (i.e. non-radial) case in two dimensions, as well as to consider the one-dimensional problem. 

The strategy of proof will be to establish a virial/Morawetz estimate for solutions to \eqref{nls} from which we may deduce scattering.  The requisite coercivity in the virial/Morawetz estimate follows from the sub-threshold assumptions \eqref{sub1} and \eqref{sub2} (specifically through the use of the sharp Gagliardo--Nirenberg inequality).  The radial assumption is used to get uniform control over error terms stemming from the spatial truncation, which is in turn needed to render the virial/Morawetz quantities finite.  In particular, we utilize the radial Sobolev embedding estimate to deal with errors at large radii.

\section{Preliminaries}
We use the standard notation $A\lesssim B$ to denote $A\leq CB$ for some $C>0$, with dependence on parameters indicated via subscripts.  We also use the `big O' notation $\mathcal{O}$.  When necessary we will write $C(A)$ to denote a positive constant depending on a parameter $A$.

 We write $a\pm$ to denote $a\pm\eps$ for sufficiently small $\eps>0$.  We employ the standard Lebesgue and Sobolev spaces, including mixed space-time Lebesgue norms.  We write $r'$ for the H\"older dual of $r$, i.e. the solution to $\tfrac{1}{r}+\tfrac{1}{r'}=1$. 

We utilize the following radial Sobolev embedding estimate, which follows from the fundamental theorem of calculus and Cauchy--Schwarz. 
\begin{lemma}[Radial Sobolev embedding]\label{L:Sobolev}  If $f\in H^1(\R^2)$ is spherically-symmetric, then
\[
\| |x|^{\frac12}f\|_{L^\infty(\R^2)} \lesssim \|f\|_{H^1(\R^2)}. 
\]
\end{lemma}

\subsection{Linear estimates; local theory}

We recall the standard dispersive estimate
\[
\|e^{it\Delta}\|_{L^{r'}(\R^2)\to L^{r}(\R^2)} \lesssim |t|^{-1+\frac{2}{r}},\quad 2 \leq r\leq\infty,
\]
which in turn yield the following Strichartz estimates \cite{GinVel, KeeTao, Str}: for $t\in I\subset\R$ and $2< q_j\leq\infty$ satisfying $\tfrac{1}{q_j}+\tfrac{1}{r_j}=\tfrac12$ for $j=1,2$, 

\begin{align*}
\|e^{it\Delta}f\|_{L_t^{q_1} L_x^{r_1}(I\times\R^2)} & \lesssim \|f\|_{L^2(\R^2)}, \\
\biggl\| \int_\R e^{-is\Delta} F(s)\,ds\biggr\|_{L^2(\R^2)} & \lesssim \|F\|_{L_t^{q_1'} L_x^{r_1'}(\R\times\R^2)},\\
\biggl\| \int_0^t e^{i(t-s)\Delta}F(s)\,ds\biggr\|_{L_t^{q_1} L_x^{r_1}(I\times\R^2)} & \lesssim \|F\|_{L_t^{q_2'} L_x^{r_2'}(I\times\R^2)}.
\end{align*}

The endpoint case $(q,r)=(2,\infty)$ may also be included in the radial setting \cite{Tao2d}, although we will not need it here.

Local well-posedness for \eqref{nls} follows from standard arguments using Strichartz estimates and Sobolev embedding.  In particular, any $u_0\in H^1$ leads to a local-in-time solution in $C_t H^1$, which may be extended to a global solution provided the $H^1$-norm remains uniformly bounded in time. See \cite{Cazenave} for a textbook treatment. 

\subsection{Variational analysis} 

We recall that $Q$ is the unique positive, radial, decaying solution to
\begin{equation}\label{elliptic}
-\Delta Q + Q - Q^{p+1} = 0,
\end{equation}
which may be constructed as an optimizer of the sharp Gagliardo--Nirenberg inequality
\[
\|f\|_{L^{p+2}(\R^2)}^{p+2} \leq C_0 \|f\|_{L^2(\R^2)}^2\|\nabla f\|_{L^2(\R^2)}^p
\]
(see e.g. \cite{Weinstein}).  Multiplying \eqref{elliptic} by $Q$ and by $x\cdot\nabla Q$ and integrating leads to the Pohozaev identities
\[
\|Q\|_{L^2}^2 = \tfrac{2}{p+2}\|Q||_{L^{p+2}}^{p+2}\qtq{and} \|\nabla Q\|_{L^2}^2 = \tfrac{p}{p+2}\|Q\|_{L^{p+2}}^{p+2}.  
\]
In particular,
\begin{equation}\label{Qeq1}
\|Q\|_{L^2}^{2}\|\nabla Q\|_{L^2}^{p-2} = \tfrac{p+2}{p}C_0^{-1}\end{equation}
and
\begin{equation}\label{Qeq2}
M(Q)^{2}E(Q)^{p-2} = (\tfrac{p-2}{2p})^{p-2}(\tfrac{p+2}{p})^2C_0^{-2}. 
\end{equation}

We will need the following two lemmas. 

\begin{lemma}[Coercivity I]\label{L:Coercive1} If 
\[
M(u_0)^{2}E(u_0)^{p-2}<(1-\delta) M(Q)^{2}E(Q)^{p-2}
\]
and
\[
\|u_0\|_{L^2}^{2}\|\nabla u_0\|_{L^2}^{p-2} \leq \|Q\|_{L^2}^{2}\|\nabla Q\|_{L^2}^{p-2},
\]
then there exists $\delta'>0$ so that
\begin{equation}\label{coer}
\|u(t)\|_{L^2}^{2}\|\nabla u(t)\|_{L^2}^{p-2} < (1-\delta')\|Q\|_{L^2}^{2}\|\nabla Q\|_{L^2}^{p-2}
\end{equation}
for all $t\in I$, where $u:I\times\R^2\to\C$ is the maximal-lifespan solution to \eqref{nls}.  In particular, $I=\R$ and $u(t)$ is uniformly bounded in $H^1$. 
\end{lemma}

\begin{proof} By the sharp Gagliardo--Nirenberg inequality and conservation of mass/energy,
\begin{align*}
(1-\delta)M(Q)^2 E(Q)^{p-2}& \geq M(u)^2 E(u)^{p-2} \\
& \geq \bigl[\tfrac12 \|u\|_{L^2}^{\frac{4}{p-2}}\|\nabla u\|_{L^2}^2 - \tfrac{C_0}{p+2}\|u\|_{L^2}^{\frac{2p}{p-2}}\|\nabla u\|_{L^2}^p\bigr]^{p-2}
\end{align*}
for all $t\in I$.  Using \eqref{Qeq2}, a short computation reveals that this is equivalent to
\[
1-\delta \geq \frac{p}{p-2} \biggl(\frac{\|u(t)\|_{L^2}^{\frac{2}{p-2}}\|\nabla u(t)\|_{L^2}}{\|Q\|_{L^2}^{\frac{2}{p-2}}\|\nabla Q\|_{L^2}} \biggr)^2 - \frac{2}{p-2}\biggl(\frac{\|u(t)\|_{L^2}^{\frac{2}{p-2}}\|\nabla u(t)\|_{L^2}}{\|Q\|_{L^2}^{\frac{2}{p-2}}\|\nabla Q\|_{L^2}} \biggr)^p.\
\]
for all $t\in I$.  The desired bound now follows from a continuity argument and the fact that for $p>2$ we have 
\[
 \tfrac{p}{p-2}y^2 - \tfrac{2}{p-2}y^p \leq 1- \delta \implies |y-1|\geq \delta' \qtq{for some}\delta'>0.
\]
Global well-posedness and uniform $H^1$ bounds now follow from the conservation of $L^2$-norm and the $H^1$ blowup criterion.  \end{proof}

\begin{lemma}[Coercivity II]\label{L:Coercive2} If 
\[
\|f\|_{L^2}^2\|\nabla f\|_{L^2}^{p-2} \leq(1-\delta)\|Q\|_{L^2}^2 \|\nabla Q\|_{L^2}^{p-2}
\]
for some $\delta>0$, then there exists $\delta'>0$ such that 
\[
\|\nabla f\|_{L^2}^2 - \tfrac{p}{p+2}\|f\|_{L^{p+2}}^{p+2} \geq \delta'\|f\|_{L^{p+2}}^{p+2}. 
\]
\end{lemma}

\begin{proof} By the sharp Gagliardo--Nirenberg inequality and \eqref{Qeq1},
\begin{align*}
E(f) & \geq \|\nabla f\|_{L^2}^2[\tfrac12-\tfrac1{p+2}C_0\|f\|_{L^2}^2\|\nabla f\|_{L^2}^{p-2}] \\
& \geq \|\nabla f\|_{L^2}^2[\tfrac12-\tfrac{1}{p+2}C_0(1-\delta)\|Q\|_{L^2}^2\|\nabla Q\|_{L^2}^{p-2}] \\
& = \|\nabla f\|_{L^2}^2[\tfrac{p-2}{2p}+\tfrac{\delta}{p}]. 
\end{align*}
Thus
\[
\tfrac{1}{p}\|\nabla f\|_{L^2}^2 -\tfrac{1}{p+2}\|f\|_{L^{p+2}}^{p+2} = E(f) - \tfrac{p-2}{2p}\|\nabla f\|_{L^2}^2 \geq \tfrac{\delta}{p}\|\nabla f\|_{L^2}^2, 
\]
which implies the result after some rearranging. 
\end{proof}
\section{Virial/Morawetz estimate}\label{S:Morawetz}

Let $u_0$ satisfy \eqref{sub1} and \eqref{sub2}, and let $u$ be the corresponding global-in-time solution to \eqref{nls} guaranteed by Lemma~\ref{L:Coercive1}.  In particular, $u$ is uniformly bounded in $H^1$ and obeys \eqref{coer}. Applying the scaling \eqref{scaling}, we may assume 
\begin{equation}\label{E0}
M(u)=E(u)=E_0.
\end{equation}

We will prove the following space-time estimate. 

\begin{proposition}\label{P:Morawetz}  For any $T>0$,
\[
\int_0^T \int |u(t,x)|^{p+2}\,dx\,dt \lesssim_{E_0} T^{\alpha},\qtq{where}\alpha=\max\{\tfrac13,\tfrac{2}{p+2}\}. 
\]
\end{proposition}

\begin{proof} We let $\phi$ be a smooth radial function satisfying
\[
\phi(x) = \begin{cases} 1 & 0\leq |x|\leq 1, \\ 0 & |x|>2.\end{cases}
\]
We may write $\phi=\phi(r)$ where $r=|x|$. We use $'$ or $\partial_r$ to denote radial derivatives.

We then set
\[
\psi(x) = \tfrac{1}{|x|}\int_0^{|x|}\phi(\rho)\,d\rho.
\]
In particular, $\psi(r)=\phi(r)$ for $r\leq 1$.  We also have the following bound:
\begin{equation}\label{psi-bds}
\begin{aligned}
|\psi(x)| &\lesssim \min\{1,\tfrac{1}{|x|}\}. 
\end{aligned}
\end{equation}

Note that
\begin{equation}\label{psi-eqs}
\begin{aligned}
r\psi'(r) &= \phi(r)-\psi(r).
\end{aligned}
\end{equation}
In particular, $\psi'(r)=\phi'(r)=0$ for $r\leq 1$, while we have
\begin{equation}\label{psi'-bds}
|\psi'(x)|\lesssim\tfrac{1}{|x|^2} \qtq{for}|x|>1.
\end{equation}

Given $R\geq 1$, we define the Morawetz quantity
\[
A(t) = \int \psi(\tfrac{x}{R})x\cdot\Im[\bar u \nabla u]\,dx
\]
which satisfies
\begin{equation}\label{Morub}
|A(t)| \lesssim R E_0\qtq{uniformly over}t\in\R. 
\end{equation}
Using \eqref{nls}, we compute
\begin{align}
\tfrac{dA}{dt} & = \Re\int \psi(\tfrac{x}{R})x_k\bigl[\bar u u_{jjk} - \bar u_{jj} u_k\bigr]\,dx \label{Mor1} \\
& \quad + \Re\int\psi(\tfrac{x}{R})x_k[\bar u\partial_k(|u|^p u)-|u|^p \bar u u_k\bigr]\,dx,\label{Mor2}
\end{align}
where subscripts denote partial derivatives and repeated indices are summed. 

For \eqref{Mor1}, we begin by observing that
\[
\Re[\bar u u_{jjk} - \bar u_{jj} u_k] = \tfrac12\partial_{jjk}|u|^2 - 2\Re\partial_j[\bar u_j u_k]. 
\]
We will insert this identity into \eqref{Mor1} and integrate by parts.  Using \eqref{psi-eqs} as well, this yields
\[
\eqref{Mor1} = -\tfrac12 \int \Delta[\psi(\tfrac{x}{R})+\phi(\tfrac{x}{R})] |u|^2\,dx + 2\int \psi(\tfrac{x}{R})|\nabla u|^2 + \psi'(\tfrac{x}{R})\tfrac{|x|}{R}|\partial_r u|^2\,dx. 
\]
Now observe
\[
\Delta[\psi(\tfrac{x}{R})+\phi(\tfrac{x}{R})] = \tfrac{1}{R^2}\phi''(\tfrac{x}{R}) + \tfrac{1}{R|x|}\bigl[2\phi'(\tfrac{x}{R})-\psi'(\tfrac{x}{R})\bigr].
\]
Recalling \eqref{psi-eqs} and \eqref{psi'-bds}, we deduce
\begin{align}
\eqref{Mor1} &= 2\int \phi(\tfrac{x}{R})|\nabla u|^2 - \tfrac{1}{2R^2}\int \phi''(\tfrac{x}{R})|u|^2 -\int\tfrac{1}{2R|x|} [2\phi'(\tfrac{x}{R})-\psi'(\tfrac{x}{R})]|u|^2\,dx \nonumber \\
& = 2\int \phi(\tfrac{x}{R})|\nabla u|^2 + \mathcal{O}(\tfrac{1}{R^2}\|u\|_{L^2}^2). \label{Mor11}
\end{align}

We turn to \eqref{Mor2}. As
\[
\Re\{\bar u \partial_k(|u|^p u) - |u|^p \bar  u u_k\} = \tfrac{p}{p+2}\partial_k(|u|^{p+2}),
\]
we may use \eqref{psi-eqs} to write
\begin{equation}\label{Mor21}
\eqref{Mor2} = -\int [\psi(\tfrac{x}{R})+\phi(\tfrac{x}{R})]\tfrac{p}{p+2}|u|^{p+2}\,dx. 
\end{equation}

We now collect \eqref{Mor11} and \eqref{Mor21} to obtain 
\begin{align}\label{Mor-wait}
\tfrac{dA}{dt} \geq  \int 2\phi(\tfrac{x}{R})|\nabla u|^2 - \tfrac{p}{p+2}[\psi(\tfrac{x}{R})+\phi(\tfrac{x}{R})]|u|^{p+2}\,dx - \mathcal{O}(\tfrac{1}{R^2})
\end{align}
Now, by construction (see e.g. \eqref{psi-bds}) and radial Sobolev embedding (Lemma~\ref{L:Sobolev}), we can estimate
\begin{equation}\label{radial-sobolev-trick}
\begin{aligned}
\int |\psi(\tfrac{x}{R})-\phi(\tfrac{x}{R})|\,|u|^{p+2}\,dx & \lesssim \int_{|x|>R}\tfrac{R}{|x|}|u|^{p+2}\,dx  \\
& \lesssim R^{-\frac{p}{2}}\||x|^{\frac12} u\|_{L^\infty}^p\|u\|_{L^2}^2 \lesssim R^{-\frac{p}{2}},
\end{aligned}
\end{equation}
and so we may continue from \eqref{Mor-wait} to get
\begin{equation}\label{Mor-wait2}
\tfrac{dA}{dt} \geq 2\int \phi(\tfrac{x}{R})\bigl[|\nabla u|^2 - \tfrac{p}{p+2}|u|^{p+2}\bigr]\,dx - \mathcal{O}(\tfrac{1}{R^\sigma}), 
\end{equation}
where
\[
\sigma=\min\{2,\tfrac{p}{2}\}. 
\]

Next, let us establish a lower bound.
\begin{lemma}\label{lower-bound} There exists $\eta>0$ such that for $R$ sufficiently large, we have 
\[
2\int\phi(\tfrac{x}{R})[|\nabla u|^2 - \tfrac{p}{p+2}|u|^{p+2} ]\,dx \geq\eta \int_{|x|\leq \frac{R}{2}} |u|^{p+2}\,dx  - \mathcal{O}(\tfrac{1}{R^\sigma})
\]
uniformly in time.
\end{lemma}

\begin{proof}[Proof of Lemma~\ref{lower-bound}] Let us write $\phi(\tfrac{x}{R})=\chi_R^2(x)$.  We first wish to show
\begin{equation}\label{supt}
\sup_t \| \chi_R u(t)\|_{L^2}^2 \|\chi_R u(t)\|_{\dot H^1}^{p-2} <(1-\tfrac12\delta')\|Q\|_{L^2}^2\|Q\|_{\dot H^1}^{p-2}
\end{equation}
for $R$ sufficiently large, where $\delta'$ is as in \eqref{coer}. Given that \eqref{coer} holds and the cutoff only decreases the $L^2$ norm, we need only consider the $\dot H^1$ norm.  For this we compute
\[
\int \chi_R^2 |\nabla u|^2 \,dx = \int |\nabla(\chi_R u)|^2 + \chi_R \Delta [\chi_R] |u|^2\,dx,
\]
which shows
\[
\|\chi_R u\|_{\dot H^1}^2 \leq \|u\|_{\dot H^1}^2 + \mathcal{O}(R^{-2}\|u\|_{L^2}^2),
\]
and hence \eqref{supt} holds for $R$ large enough.

Using Lemma~\ref{L:Coercive2}, \eqref{supt} implies
\[
\|\chi_R u(t)\|_{\dot H^1}^2 - \tfrac{p}{p+2}\|\chi_R u(t)\|_{L^{p+2}}^{p+2}\geq 10\eta\|\chi_R u(t)\|_{L^{p+2}}^{p+2}
\]
for some $\eta>0$, uniformly in $t$.  Now, the argument just given above shows that we may replace 
\[
\|\chi_R u(t)\|_{\dot H^1}^2 \qtq{with} \int \phi(\tfrac{x}{R})|\nabla u|^2\,dx
\]
with errors that are $\mathcal{O}(R^{-2})$ uniformly in time.  Similarly, estimating as in \eqref{radial-sobolev-trick}, we may write
\begin{align*}
\int |\chi_R u|^{p+2} & = \int \phi(\tfrac{x}{R})|u|^{p+2} + \mathcal{O}\biggl[\int_{|x|>\frac{R}{2}} |u|^{p+2}\,dx\biggr] \\
& = \int \phi(\tfrac{x}{R})|u|^{p+2}\,dx + \mathcal{O}(R^{-\frac{p}{2}})
\end{align*}
uniformly in $t$.   This completes the proof. 
\end{proof} 

With Lemma~\ref{lower-bound} in place, we may combine \eqref{Mor-wait2}, the fundamental theorem of calculus, and \eqref{Morub} to deduce
\[
\int_0^T \int_{|x|\leq \frac{R}{2}} |u(t,x)|^{p+2}\,dx\,dt \lesssim_{E_0} R+\tfrac{T}{R^\sigma}
\]
uniformly in $T,R\geq 1$.  As radial Sobolev embedding (Lemma~\ref{L:Sobolev}) yields
\[
\int_{|x|>\frac{R}{2}}|u(t,x)|^{p+2}\,dx \lesssim R^{-\frac{p}{2}} \| |x|^{\frac12} u\|_{L^\infty}^p \|u\|_{L^2}^2 \lesssim R^{-\frac{p}{2}}\|u\|_{H^1}^{p+1} \lesssim R^{-\frac{p}{2}}
\]
uniformly in time, we deduce
\[
\int_0^T \int |u(t,x)|^{p+2}\,dx\,dt \lesssim_{E_0} R+TR^{-\sigma}
\]
uniformly in $T,R\geq 1$.  Choosing $R=T^{\frac{1}{1+\sigma}}$ yields
\[
\int_0^T \int |u(t,x)|^{p+2}\,dx \,dt \lesssim T^{\frac{1}{1+\sigma}}.
\]
Observing that 
\[
\tfrac{1}{1+\sigma}=\tfrac{1}{1+\min\{2,\frac{p}{2}\}} = \max\{\tfrac13,\tfrac{2}{p+2}\}=\alpha,
\] 
we complete the proof of the proposition. \end{proof}

\section{Scattering} 

In this section we will use the Morawetz/virial estimate of Proposition~\ref{P:Morawetz} to establish scattering.

\begin{proof}[Proof of Theorem~\ref{T}] For initial data $u_0$ as in Theorem~\ref{T}, we are guaranteed a global-in-time solution $u(t)$ satisfying uniform $H^1$ bounds by Lemma~\ref{L:Coercive1}.  We rescale $u$ so that \eqref{E0} holds, and we have the Morawetz/virial estimate, Proposition~\ref{P:Morawetz}.

Our argument is similar to the one appearing in \cite{DM2}.  Let $\eps>0$ be a small parameter to be chosen sufficiently small (depending on $E_0$) below.  As Sobolev embedding and Strichartz yield
\[
\|e^{it\Delta}u_0\|_{L_{t,x}^{2p}(\R\times\R^2)} \lesssim \|u_0\|_{H^1(\R^2)} \lesssim_{E_0} 1,
\]
we may split $\R$ into $J=J(\eps,E_0)$ intervals $I_j$ such that
\begin{equation}\label{dispersed}
\|e^{it\Delta}u_0\|_{L_{t,x}^{2p}(I_j\times\R^2)} <\eps
\end{equation}
for each $j$.  We let $T=T(\eps)$ be a large parameter to be determined below.  We will prove
\begin{equation}\label{WWP}
\|u\|_{L_{t,x}^{2p}(I_j\times\R^2)}^{2p} \lesssim_{E_0} T
\end{equation}
for each $j$.  Then, summing over $j$ yields the critical global space-time bound
\[
\|u\|_{L_{t,x}^{2p}(\R\times\R^2)}^{2p} \lesssim_{E_0} T,
\]
which in turn yields scattering by standard arguments. 

As H\"older's inequality and Sobolev embedding imply
\begin{equation}\label{HolderT}
\|u\|_{L_{t,x}^{2p}(I\times\R^2)}^{2p} \lesssim_{E_0} \langle I\rangle 
\end{equation}
for any interval $I\subset\R$, it suffices to consider $j$ such that $|I_j|>2T$. 

Let us fix one such interval, say $I=(a,b)$ with $|I|>2T$.  We will show that there exists $t_1\in(a,a+T)$ such that 
\begin{equation}\label{scatter1}
\biggl\| \int_0^{t_1}e^{i(t-s)\Delta}(|u|^p u)(s)\,ds\biggr\|_{L_{t,x}^{2p}([t_1,\infty)\times\R^2)} \lesssim_{E_0} C(\eps) T^{-\beta}+ \eps^{\frac12}
\end{equation}
for some $\beta>0$.

Assuming \eqref{scatter1} for the moment, let us complete the proof.  We use the Duhamel formula to write
\[
e^{i(t-t_1)\Delta}u(t_1) = e^{it\Delta}u_0 + i\int_0^{t_1} e^{i(t-s)\Delta}(|u|^p u)(s)\,ds.
\]
Thus, choosing $T$ sufficiently large depending on $\eps$ and recalling \eqref{dispersed}, we deduce
\[
\|e^{i(t-t_1)\Delta}u(t_1)\|_{L_{t,x}^{2p}([t_1,b]\times\R^2)} \lesssim \eps^{\frac12}.
\]
For $\eps$ small enough, this yields by a continuity argument the bound
\[
\|u\|_{L_{t,x}^{2p}([t_1,b]\times\R^2)} \lesssim \eps^{\frac12}.
\]
On the other hand, using $|t_1-a|<T$ and \eqref{HolderT}, we get
\[
\|u\|_{L_{t,x}^{2p}([a,t_1]\times\R^2)}^{2p}\lesssim T,
\]
and hence \eqref{WWP} holds, as desired. 

It remains to prove \eqref{scatter1}. By time-translation invariance, we may assume $a=0$.  We begin by applying Proposition~\ref{P:Morawetz}, which yields 
\begin{equation}\label{scattering1}
\int_{0}^T\int |u(t,x)|^{p+2}\,dx\,dt \lesssim_{E_0} T^{\alpha},\qtq{where}\alpha=\max\{\tfrac13,\tfrac{2}{p+2}\}.
\end{equation}
We claim that there exists $t_0\in[\tfrac{T}{4},\tfrac{T}{2}]$ and $\delta=\delta(\eps,E_0)>0$ such that
\begin{equation}\label{scattering2}
\int_{t_0}^{t_0+\delta T^{1-\alpha}} \int |u(t,x)|^{p+2}\,dx\,dt < \eps. 
\end{equation}
Indeed, as $[\tfrac{T}{4},\tfrac{T}{2}]$ is covered by $\sim\delta^{-1} T^\alpha$ intervals of length $\delta T^{1-\alpha}$, the bound \eqref{scattering1} shows that there must be some interval $[t_0,t_0+\delta T^{1-\alpha}]$ obeying
\[
\int_{t_0}^{t_0+\delta T^{1-\alpha}} \int |u(t,x)|^{p+2}\,dx\,dt\lesssim \delta C(E_0),
\]
which yields the claim.

We now set 
\[
t_1=t_0+\delta T^{1-\alpha}
\]
and observe that since $t_0\leq \tfrac{T}{2}$, we may guarantee that $t_1<T$.

We will estimate the integral in \eqref{scatter1} by estimating separately the contribution of $[0,t_0]$ and $[t_0,t_1]$. 

We first treat $[0,t_0]$. For $t>t_1$,  we may use the dispersive estimate, H\"older's inequality, and \eqref{scattering1} to estimate
\begin{align*}
\biggl\| \int_0^{t_0} e^{i(t-s)\Delta}|u|^p u\,ds\biggr\|_{L_x^\infty} & \lesssim \int_0^{t_0} |t-s|^{-1} \|u(s)\|_{L^{p+2}}^{\frac{(p+2)(p-1)}{p}}\|u(s)\|_{L^2}^{\frac{2}{p}}\,ds \\
& \lesssim\biggl[\int_0^{t_0}\int |u(s,x)|^{p+2}\,dx\,dt\biggr]^{\frac{p-1}{p}} \|\,|t-s|^{-1}\|_{L_s^p} \\
& \lesssim_{E_0} T^{\frac{p-1}{p}\alpha}|t-t_0|^{-\frac{p-1}{p}} \lesssim \delta^{-\frac{p-1}{p}}T^{-(1-2\alpha)\frac{p-1}{p}} 
\end{align*}
yielding
\[
\biggl\| \int_0^{t_0} e^{i(t-s)\Delta}|u|^p u\,ds\biggr\|_{L_{t,x}^\infty([t_1,\infty)\times\R^2)} \lesssim \delta^{-\frac{p-1}{p}}T^{-(1-2\alpha)\frac{p-1}{p}}.
\]
On the other hand, we may write
\[
i\int_0^{t_0}e^{i(t-s)\Delta}|u|^p u\,ds = e^{i(t-t_0)\Delta} u(t_0)-e^{it\Delta}u_0,
\]
so that by Strichartz we have
\[
\biggl\| \int_0^{t_0} e^{i(t-s)\Delta}|u|^p u(s)\,ds \biggr\|_{L_{t,x}^4(\R\times\R^2)}\lesssim_{E_0} 1. 
\]
Thus, by interpolation and the fact that $\tfrac13<\alpha<\tfrac12$, we get
\begin{equation}\label{NTS-1}
\begin{aligned}
\biggl\|\int_0^{t_0} &e^{i(t-s)\Delta}|u|^p u(s)\,ds \biggr\|_{L_{t,x}^{2p}([t_1,\infty)\times\R^2)}\\ 
& \lesssim_{E_0} [ \delta^{-\frac{p-1}{p}}T^{-(1-2\alpha)\frac{p-1}{p}}]^{1-\frac{2}{p}} \lesssim_{E_0} C(\eps) T^{-\beta}
\end{aligned}
\end{equation}
for some $\beta>0$.  This is an acceptable contribution to \eqref{scatter1}. 

We next consider the contribution of $[t_0,t_1]$. Let us first show how the estimate works, employing the $a\pm$ notation; we will show how to choose exponents more precisely in Remark~\ref{remark-exponents} below. By Sobolev embedding, Strichartz, the fractional chain rule, and \eqref{scattering2},  
\begin{align*}
\biggl\| \int_{t_0}^{t_1} e^{i(t-s)\Delta}|u|^p u(s)\,ds\biggr\|_{L_{t,x}^{2p}}& \lesssim \biggl\| |\nabla|^{s_c}\int_{t_0}^{t_1}e^{i(t-s)\Delta}|u|^p u(s)\,ds\biggr\|_{L_t^{2p}L_x^\frac{2p}{p-1}} \\
& \lesssim \| |\nabla|^{s_c}(|u|^p u)\|_{L_t^{2-}L_x^{1+}([t_0,t_1]\times\R^2)} \\
& \lesssim \|u\|_{L_{t,x}^{p+2}([t_0,t_1]\times\R^2)}^{\frac{p+2}{2}+}\|u\|_{L_t^\infty L_x^{\infty-}}^{\frac{p-2}{2}-}\||\nabla|^{s_c}u\|_{L_t^\infty L_x^{2+}} \\
& \lesssim_{E_0} \eps^{\frac12},
\end{align*}
which is again an acceptable contribution to \eqref{scatter1}.  

This completes the proof of \eqref{scatter1} and hence of Theorem~\ref{T}.\end{proof}

\begin{remark}\label{remark-exponents} It is also possible (although, we contend, less transparent) to choose the exponents in the final estimate above more precisely:  Let $\theta\in(0,1)$ be a small parameter satisfying $p(1-\theta)>2$.  Then we may estimate   
\begin{align*}
 \| |\nabla|^{s_c}(|u|^p u)\|_{L_t^{2-\theta}L_x^{\frac{4-2\theta}{4-3\theta}}([t_0,t_1]\times\R^2)} 
& \lesssim \|u\|_{L_{t,x}^{p+2}([t_0,t_1]\times\R^2)}^{\frac{p+2}{2-\theta}}\|u\|_{L_t^\infty L_x^{r_2}}^{\frac{p(1-\theta)-2}{2-\theta}}\||\nabla|^{s_c}u\|_{L_t^\infty L_x^{r_1}},
\end{align*}
where, given a choice of $r_1$ we must have (by scaling)
\[
r_2 = \tfrac{2r_1[p(1-\theta)-2]}{r_1(2-3\theta)-(4-2\theta)}.
\]
In particular, to guarantee finiteness of $r_2$ we should take 
\begin{equation}\label{r1-constraint1}
r_1>\tfrac{4-2\theta}{2-3\theta}>2,
\end{equation}
which is compatible with $r_1\leq\tfrac{2p}{p-2}$ (needed for the embedding $\dot H^{s_c,r_1}\subset H^1$) provided $\theta$ also obeys $\theta\leq\tfrac{2}{p+1}$. It then remains to verify that we may guarantee $r_2\geq 2$. After some rearranging, this reduces to the constraint
\[
r_1\leq \begin{cases} \tfrac{4-2\theta}{2-3\theta -[p(1-\theta)-2]} & \text{if } p(1-\theta)-2<2-3\theta \\ \infty & \text{otherwise}.\end{cases}
\]
As this is compatible with \eqref{r1-constraint1}, we conclude that there exist suitable choices of exponents, as claimed. 
\end{remark}

\subsection*{Acknowledgements} B.~D.~  was supported by NSF DMS-1764358 and completed part of this work while a von Neumann fellow at the Institute for Advanced Study.


\end{document}